\documentclass[10pt]{amsart}

\usepackage{graphicx}
\usepackage{amsmath}
\usepackage{amssymb}

\def\RR{{\mathbb R}}
\def\MM{{\mathbb M}}

\def\NN{{\mathbb N}}
\def\OO{{\mathbb O}}

\def\Z{\mathcal{Z}}
\def\R{\mathcal{R}}
\def\Q{\mathcal{Q}}

\def\det{{\rm det}}

\def\Aff{{\rm Aff}}

\def\eps{\varepsilon}
\def\then{\Rightarrow}
\def\inff{\mathop{\inf\limits_{a\in\RR^N}}}
\def\infff{\mathop{\inff\limits_{b\in\RR^m}}}

\title[Relaxation and 3d-2d passage with determinant type constraints]{Relaxation and 3d-2d passage with determinant type constraints: an outline}

\author{Omar Anza Hafsa}
\address{UNIVERSITE MONTPELLIER II, UMR-CNRS 5508, LMGC,  Place Eug\`ene Bataillon, 34095 Montpellier, France.}
\email{Omar.Anza-Hafsa@univ-montp2.fr}
\author{Jean-Philippe Mandallena}
\address{UNIVERSITE DE NIMES, Site des Carmes, Place Gabriel P\'eri, 30021 N\^\i mes, France.}
\email{jean-philippe.mandallena@unimes.fr}

\begin{document}

\begin{abstract}
We outline our work (see \cite{oah-jpm06,oah-jpm07,oah-jpm08b,oah-jpm08a}) on relaxation and 3d-2d passage with determinant type constraints. Some open questions are addressed. This outline-paper comes as a companion to \cite{oah-jpm09}.
\end{abstract}

\maketitle

\tableofcontents

\section{Relaxation with determinant type constraints}

\subsection{Statement of the problem}\ 

Let $m,N\in\NN$ (with $\min\{m,N\}>1$), let $p>1$ and let $W:\MM^{m\times N}\to[0,+\infty]$ be Borel measurable and $p$-coercive, i.e., 
$$
\exists C>0\ \forall F\in\MM^{m\times N}\ W(F)\geq C|F|^p,
$$
 where $\MM^{m\times N}$ denotes the space of real $m\times N$ matrices. Define the functional $I:W^{1,p}(\Omega;\RR^m)\to[0,+\infty]$ by
$$
\displaystyle I(\phi):=\int_\Omega W(\nabla\phi(x))dx,
$$
where $\Omega\subset\RR^N$ is a bounded open set, and consider  $\overline{I}:W^{1,p}(\Omega;\RR^m)\to[0,+\infty]$ (the relaxed functional of $I$) given by 
$$
\displaystyle\overline{I}(\phi):=\inf\left\{\liminf_{n\to+\infty}I(\phi_n):\phi_n\stackrel{L^p}{\to}\phi\right\}.
$$
Denote the quasiconvex envelope of $W$ by $\Q W:\MM^{m\times N}\to[0,+\infty]$. The problem of the relaxation is the following:
\begin{itemize}
\item[($\mathcal{P}_1$)] {\em prove (or disprove) that 
$$
\forall\phi\in W^{1,p}(\Omega;\RR^m)\ \ \overline{I}(\phi)=\int_\Omega \Q W(\nabla\phi(x))dx
$$
and find a representation formula for $\Q W$.}
\end{itemize}
At the begining of the eighties, Dacorogna answered to ($\mathcal{P}_1$) in the case where $W$ is ``finite and without singularities" (see \S 1.2). Recently, we extended the Dacorogna theorem as Theorem A and Theorem B (see \S 1.3 and \S 1.4) and we showed that these theorems can be used to deal with ($\mathcal{P}_1$) under the ``weak-Determinant Constraint", i.e., when $m=N$ and $W:\MM^{N\times N}\to[0,+\infty]$ is compatible with the following two conditions:
$$
\hbox{(w-DC)}\left\{\begin{array}{l}W(F)=+\infty\iff -\delta\leq \det F\leq 0\hbox{ with $\delta\geq 0$ (possibly very large)}\\
W(F)\to+\infty\hbox{ as }\det F\to 0^+\end{array}\right.
$$
(see \S 1.6). However, the results of this section do not allow to treat ($\mathcal{P}_1$) under the ``strong-Determinant Constraint", i.e., when $m=N$ and $W:\MM^{N\times N}\to[0,+\infty]$ is compatible with the two basic conditions of nonlinear elasticity:
$$
\hbox{(s-DC)}\left\{\begin{array}{ll}W(F)=+\infty\iff \det F\leq 0&\hbox{ (non-interpenetration of matter)}\\
W(F)\to+\infty\hbox{ as }\det F\to 0^+&\left(\begin{array}{l}\hbox{necessity of an infinite amount}\\
\hbox{of energy to compress a finite}\\
\hbox{volume into zero volume}\end{array}\right)\end{array}\right.
$$
(see \S 1.7).

\subsection{Representation of $\Q W$  and $\overline{I}$: finite case}\ 

Let $\Z_\infty W, \Z W:\MM^{m\times N}\to[0,+\infty]$ be respectively defined by:
\begin{itemize}
\item[\SMALL$\blacklozenge$] $\displaystyle \Z_\infty W(F):=\inf\left\{\int_Y W(F+\nabla\varphi(y))dy:\varphi\in W^{1,\infty}_0(Y;\RR^m)\right\}$;
\item[\SMALL$\blacklozenge$] $\displaystyle\Z W(F):=\inf\left\{\int_Y W(F+\nabla\varphi(y))dy:\varphi\in\Aff_0(Y;\RR^m)\right\}$,
\end{itemize}
where $Y:=]0,1[^N$, $W^{1,\infty}_0(Y;\RR^m):=\{\varphi\in W^{1,\infty}(Y;\RR^m):\varphi=0\hbox{ on }\partial Y\}$ and $\Aff_0(Y;\RR^m):=\{\varphi\in \Aff(Y;\RR^m):\varphi=0\hbox{ on }\partial Y\}$ with $\Aff(Y;\RR^m)$ denoting the space of continuous piecewise affine functions from $Y$ to $\RR^m$.
\newtheorem*{Remark}{{\rm\em Remark}\rm}
\begin{Remark}
\rm One always has $W\geq\Z W\geq \Z_\infty W\geq\Q W$.
\end{Remark}
\newtheorem*{Theorem}{\bf Theorem}
\begin{Theorem}[Dacorogna \cite{dacorogna82} 1982]\ \hskip150mm
\begin{itemize}
\item[(a)] Representation of $\Q W${\rm:} if $W$ is continuous and finite then 
$$\Q W=\Z W=\Z_\infty W.$$
\item[(b)] Integral representation of $\overline{I}${\rm:} if $W$ is continuous and
$$
\exists c>0\ \ \forall F\in\MM^{m\times N}\ \ W(F)\leq c(1+|F|^p)
$$
then 
$$
\forall\phi\in W^{1,p}(\Omega;\RR^m)\quad \displaystyle\overline{I}(\phi)=\int_\Omega \Q W(\nabla\phi(x))dx.
$$
\end{itemize}
\end{Theorem}

\subsection{Representation of $\Q W$: non-finite case}\ 

The part (a) of the Dacorogna theorem can be extended as follows.

\newtheorem*{ThA}{\bf Theorem A}
\begin{ThA}[see \cite{oah-jpm07,oah-jpm08b,oah-jpm09}]\ \hskip150mm
\begin{itemize}
\item[\small$\blacktriangleright$] If $\Z_\infty W$ is finite then $\Q W=\Z_\infty W$.
\item[\small$\blacktriangleright$] If $\Z W$ is finite then $\Q W=\Z W=\Z_\infty W$.

\end{itemize}
\end{ThA}

\begin{proof}
\small We need (the two last assertions, the first one being used at the end of \S1.3, of)  the following result.
\begin{Theorem}[Fonseca \cite{fonseca88} 1988]\ \hskip150mm
\begin{itemize}
\item[(1)] If $\Z_{\infty} W$ (resp. $\Z W$)  is finite then $\Z_{\infty} W$ (resp. $\Z W$)  is rank-one convex. 
\item[(2)] If $\Z_{\infty} W$ (resp. $\Z W$) is finite then $\Z_{\infty} W$ (resp. $\Z W$) is continuous.
\item[(3)] $\Z_\infty W\leq \Z\Z_\infty W$ and $\Z \Z W=\Z W$.
\end{itemize}
\end{Theorem}
One always has $W\geq\Z W\geq \Z_\infty W\geq\Q W$. Hence:
\begin{itemize}
\item[(i)] $\Q\Z_\infty W=\Q W\leq\Z_\infty W$;
\item[(ii)] $\Q\Z W=\Q\Z_\infty W=\Q W$.
\end{itemize}

\small$\blacktriangleright$ If $\Z_\infty W$ is finite then $\Z_\infty W$ is continuous by the property (2) of Fonseca. From the first part of the Dacorogna theorem it follows that $\Q\Z_\infty W=\Z\Z_\infty W$. But $\Z_\infty W\leq\Z\Z_\infty W$ by the property (3) of Fonseca, and so $\Q W=\Z_\infty W$ by using (i).

\medskip

\small$\blacktriangleright$ If $\Z W$ is finite then also is $\Z_\infty W$. Hence $\Q W=\Z_\infty W$ by the previous reasoning. On the other hand, $\Z W$ is continuous by the property (2) of Fonseca. From the first part of the Dacorogna theorem it follows that $\Q\Z W=\Z\Z W$. But $\Z\Z W=\Z W$ by the property (3) of Fonseca, and so $\Q W=\Z W$ by using (ii).
\end{proof}

\newtheorem*{Question}{\bf Question}
\begin{Question}
Prove (or disprove) that if $\Z_\infty W$ is finite, also is $\Z W$.
\end{Question}

\subsection{Representation of $\overline{I}$: non-finite case}\ 

The  part (b) of the Dacorogna theorem can be extended as follows.

\newtheorem*{ThB}{\bf Theorem B}
\begin{ThB}[see \cite{oah-jpm07,oah-jpm08b,oah-jpm09}]\ \hskip150mm
\begin{itemize}
\item[\small$\blacktriangleright$] If 
$
\exists c>0\ \ \forall F\in\MM^{m\times N}\ \ \Z_\infty W(F)\leq c(1+|F|^p)
$
then
$$
\forall \phi\in W^{1,p}(\Omega;\RR^m)\ \ \displaystyle\overline{I}(\phi)=\int_\Omega \Q W(\nabla\phi(x))dx.
$$
\item[\small$\blacktriangleright$] If 
$
\exists c>0\ \ \forall F\in\MM^{m\times N}\ \ \Z  W(F)\leq c(1+|F|^p)
$
then
$$
\forall \phi\in W^{1,p}(\Omega;\RR^m)\ \ \displaystyle\overline{I}(\phi)=\overline{I}_{\rm aff}(\phi)=\int_\Omega \Q W(\nabla\phi(x))dx
$$
with $\overline{I}_{\rm aff}:W^{1,p}(\Omega;\RR^m)\to[0,+\infty]$ defined by
$$
\overline{I}_{\rm aff}(\phi):=\inf\left\{\liminf_{n\to+\infty}I(\phi_n):\Aff(\Omega;\RR^m)\ni\phi_n\stackrel{L^p}{\to}\phi\right\}.
$$
\end{itemize}
\end{ThB}
\begin{proof}[\small\bf\em Outline of the proof] \small$\blacktriangleright$ Let $\Z_\infty I, \overline{\Z_\infty I}, \overline{\Z_\infty I}_{\rm aff}:W^{1,p}(\Omega;\RR^m)\to[0,+\infty]$ be respectively defined by:
\begin{itemize}
\item[\SMALL$\blacklozenge$] $\displaystyle\Z_\infty I(\phi):=\int_\Omega\Z_\infty W(\nabla\phi(x))dx$;
\item[\SMALL$\blacklozenge$] $\displaystyle\overline{\Z_\infty I}(\phi):=\inf\left\{\liminf_{n\to+\infty}\Z_\infty I(\phi_n):\phi_n\stackrel{L^p}{\to}\phi\right\}$;
\item[\SMALL$\blacklozenge$] $\displaystyle\overline{\Z_\infty I}_{\rm aff}(\phi):=\inf\left\{\liminf_{n\to+\infty}\Z_\infty I(\phi_n):\Aff(\Omega;\RR^m)\ni\phi_n\stackrel{L^p}{\to}\phi\right\}.$ 
\end{itemize}
Since $\Z_\infty W$ is of $p$-polynomial growth, i.e., $\exists c>0\ \forall F\in\MM^{m\times N}\ \Z_\infty W(F)\leq c(1+|F|^p)$, it follows that $\Z_\infty W$ is (finite and so) continuous by the property (2) of Fonseca. By the second part of the Dacorogna theorem we deduce that
$$
\forall\phi\in W^{1,p}(\Omega;\RR^m)\ \ \overline{\Z_\infty I}(\phi)=\int_\Omega \Q\Z_\infty W(\nabla\phi(x))dx.
$$
But one always has $\Q\Z_\infty W=\Q W$, hence
$$
\forall\phi\in W^{1,p}(\Omega;\RR^m)\ \ \overline{\Z_\infty I}(\phi)=\int_\Omega \Q W(\nabla\phi(x))dx.
$$
Thus, it suffices to prove that $\overline{I}\leq\overline{\Z_\infty I}$ (the reverse inequality being trivially true). The key point of the proof is that we can establish (by using the Vitali covering theorem and without assuming that $\Z_\infty W$ is of $p$-polynomial growth) the following lemma.
\newtheorem*{Lemma}{\bf Lemma}
\begin{Lemma}
$\overline{I}\leq\overline{\Z_\infty I}_{\rm aff}$.
\end{Lemma}
On the other hand, as $\Z_\infty W$ is of $p$-polynomial growth and $\Aff(\Omega;\RR^m)$ is strongly dense in $W^{1,p}(\Omega;\RR^m)$, it is easy to see that $\overline{\Z_\infty I}_{\rm aff}=\overline{\Z_\infty I}$, and the result follows.

\medskip

\small$\blacktriangleright$ Let $\Z I, \overline{\Z_\infty I}, \overline{\Z_\infty I}_{\rm aff}:W^{1,p}(\Omega;\RR^m)\to[0,+\infty]$ be respectively defined by:
\begin{itemize}
\item[\SMALL$\blacklozenge$] $\displaystyle\Z I(\phi):=\int_\Omega\Z W(\nabla\phi(x))dx$;
\item[\SMALL$\blacklozenge$] $\displaystyle\overline{\Z I}(\phi):=\inf\left\{\liminf_{n\to+\infty}\Z I(\phi_n):\phi_n\stackrel{L^p}{\to}\phi\right\}$;
\item[\SMALL$\blacklozenge$] $\displaystyle\overline{\Z I}_{\rm aff}(\phi):=\inf\left\{\liminf_{n\to+\infty}\Z I(\phi_n):\Aff(\Omega;\RR^m)\ni\phi_n\stackrel{L^p}{\to}\phi\right\}.$ 
\end{itemize}
As $\Z W$ is of $p$-polynomial growth and (so) continuous (by the property (2) of Fonseca), from the second part of the Dacorogna theorem (and since $\Q\Z W=\Q W$ is always true) we deduce that
$$
\forall\phi\in W^{1,p}(\Omega;\RR^m)\ \ \overline{\Z I}(\phi)=\int_\Omega \Q W(\nabla\phi(x))dx.
$$
It is then sufficient to prove that $\overline{I}_{\rm aff}\leq\overline{\Z I}$ (the inequalities $\overline{I}\leq\overline{I}_{\rm aff}$ and $\overline{\Z I}\leq\overline{I}$ being trivially true). The key point of the proof is that we can establish (by using the Vitali covering theorem and without assuming that $\Z W$ is of $p$-polynomial growth) the following lemma.
\begin{Lemma}
$\overline{I}_{\rm aff}=\overline{\Z I}_{\rm aff}$.
\end{Lemma}
On the other hand, as $\Z W$ is of $p$-polynomial growth and $\Aff(\Omega;\RR^m)$ is strongly dense in $W^{1,p}(\Omega;\RR^m)$, it is clear that $\overline{\Z I}_{\rm aff}=\overline{\Z I}$, and the result follows.
\end{proof}

We see here that the integrands $W$ for which $\Z_\infty W$ or $\Z W$ is of $p$-polynomial have a ``nice" behavior with respect to ($\mathcal{P}_1$). So, it could be interesting to introduce a new class of integrands (that we will call the class of $p$-ample\footnote{We use the term ``$p$-ample" because of some analogies with the concept (developed in differential geometry by Gromov) of amplitude of a differential relation (see \cite{gromov86} for more details).} integrands) as follows:
$$
\hbox{$W$ is $p$-ample $\iff$ $\exists c>0\ \ \forall F\in\MM^{m\times N}\ \ \Z_\infty W(F)\leq c(1+|F|^p)$.}
$$
Thus, Theorems A and B can be summarized as follows. 
\newtheorem*{ThA-B}{\bf Theorem A-B}
\begin{ThA-B}
If $W$ is $p$-ample then 
$$
\forall\phi\in W^{1,p}(\Omega;\RR^m)\ \ \overline{I}(\phi)=\int_\Omega\Q W(\nabla\phi(x))dx\hbox{ and $\Q W=\Z_\infty W$.}
$$ 
\end{ThA-B}
\begin{Question}
Prove (or disprove) that $W$ is $p$-ample if and only if $\Q W$ is of $p$-polynomial growth.
\end{Question}

An analogue result of Theorem B was proved by Ben Belgacem (who is in fact the first that obtained an integral representation for $\overline{I}$ in the non-finite case). Let  $\{\R_i W\}_{i\in\NN}$ be defined by $\R_0 W:=W$ and for each $i\in\NN^*$ and each $F\in\MM^{m\times N}$,
$$
\R_{i+1}W(F):=\infff\limits_{t\in[0,1]}\big\{(1-t)\R_i W(F-t a\otimes b)+t\R_i W(F+(1-t)a\otimes b)\big\}.
$$
By Kohn et Strang (see \cite{kohn-strang86}) we have $\R_{i+1}W\leq\R_i W$ for all $i\in\NN$ and $\R W=\inf_{i\geq 0}\R_i W$, where $\R W$ denotes the rank-one convex envelope of $W$. The Ben Belgacem theorem can be stated as follows.

\begin{Theorem}[Ben Belgacem \cite{benbelgacem96,benbelgacem00} 1996]\ \hskip150mm

Assume that{\rm:}
\begin{itemize}
\item[(BB$_1$)] $\OO_W:={\rm int}\Big\{F\in\MM^{m\times N}:\forall i\in\NN\ \Z\R_i W(F)\leq\R_{i+1}W(F)\Big\}$ is dense in $\MM^{m\times N};$
\item[(BB$_2$)] $\forall i\in\NN^*\ $ $\forall F\in\MM^{m\times N}\ $ $\forall\{F_n\}_n\subset \OO_W$ 
$$F_n\to F\then \R_i W(F)\geq\limsup\limits_{n\to+\infty}\R_i W(F_n);$$
\item[\SMALL$\blacklozenge$] $\exists c>0\ \ \forall F\in\MM^{m\times N}\ \ \R W(F)\leq c(1+|F|^p)$.
\end{itemize}
Then
$$
\forall\phi\in W^{1,p}(\Omega;\RR^m)\quad \displaystyle\overline{I}(\phi)=\int_\Omega \Q\R W(\nabla\phi(x))dx.
$$
\end{Theorem}

Generally speaking, as rank-one convexity and quasiconvexity do not coincide, Theorem B and the Ben Belgacem theorem are not identical. However, we have

\begin{Lemma} If either $\Z_\infty W$ or $\Z W$ is finite then $\Q\R W=\Q W$.
\end{Lemma}

\begin{proof}[\small\bf\em Proof]
\small If $\Z_\infty W$ (resp. $\Z W$) is finite then $\Z_\infty W$ (resp. $\Z W$) is rank-one convex by the property (1) of Fonseca. Consequently $\Z_\infty W\leq\R W$ (resp. $\Z W\leq\R W$) (and Theorem B$^\prime$ below follows by applying Theorem B). Thus, we have $\Z_\infty W\leq\R W\leq W$ (resp. $\Z_\infty W\leq\R W\leq W$), hence $\Q\Z_\infty W\leq\Q\R W\leq \Q W$ (resp. $\Q\Z W\leq\Q\R W\leq \Q W$) and so $\Q\R W=\Q W$ since one always has $\Q\Z_\infty W=\Q W$ (resp. $\Q\Z W=\Q W$).
\end{proof}

\newtheorem*{ThB-bis}{\bf Theorem B$^\prime$}
\begin{ThB-bis}
Assume that $\exists c>0\ \forall F\in\MM^{m\times N}\ \R W(F)\leq c(1+|F|^p)$. Then{\rm :}
\begin{itemize}
\item[\small$\blacktriangleright$] if $\Z_\infty W$ is finite then 
$$\displaystyle\forall \phi\in W^{1,p}(\Omega;\RR^m)\ \ \displaystyle\overline{I}(\phi)=\int_\Omega \Q W(\nabla\phi(x))dx;
$$ 
\item[\small$\blacktriangleright$] if $\Z W$ is finite then 
$$
\forall \phi\in W^{1,p}(\Omega;\RR^m)\ \ \displaystyle\overline{I}(\phi)=\overline{I}_{\rm aff}(\phi)=\int_\Omega \Q W(\nabla\phi(x))dx.
$$
\end{itemize}
\end{ThB-bis}

\begin{Question}
Prove (or disprove) that if {\rm(BB$_1$)} and {\rm(BB$_2$)} hold then $\Z W$ is finite.
\end{Question}

\subsection{Application 1: ``non-zero-Cross Product Constraint"}\ 

Consider $W_0:\MM^{3\times 2}\to[0,+\infty]$ Borel measurable and $p$-coercive and the following condition 
\begin{itemize}
\item[(P)] $\exists\alpha,\beta>0\ \forall \xi=(\xi_1\mid \xi_2)\in\MM^{3\times 2}\ \big(|\xi_1\land \xi_2|\geq\alpha\then W_0(\xi)\leq\beta(1+|\xi|^p)\big)$
\end{itemize}
with $\xi_1\land\xi_2$ denoting the cross product of vectors $\xi_1,\xi_2\in\RR^3$. When $W_0$ satisfies (P) it is compatible with the ``non-zero-Cross Product Constraint", i.e., with the following two conditions:
$$
\hbox{($*$-CPC)}\left\{\begin{array}{l}W_0(\xi_1\mid\xi_2)=+\infty\iff |\xi_1\land\xi_2|=0\\
W_0(\xi_1\mid\xi_2)\to+\infty\hbox{ as }|\xi_2\land\xi_2|\to 0.\end{array}\right.
$$
 The interest of considering ($*$-CPC) comes from the 3d-2d problem (see \S 2): if $W$ is compatible with (s-DC) then $W_0$ given by $W_0(\xi):=\inf_{\zeta\in\RR^3} W(\xi\mid\zeta)$ is compatible with ($*$-CPC). One can prove that
 $$
{\rm(P)}\then\exists c>0\ \forall F\in\MM^{3\times 2}\ \Z W(F)\leq c(1+|F|^p)
$$
(see \cite{oah-jpm07,oah-jpm08a,oah-jpm09}) which roughly means that the ``non-zero Cross Product Constraint" is $p$-ample. Applying Theorem B we obtain

\newtheorem*{CorollaryA}{\bf Corollary 1}
\begin{CorollaryA}
If $W_0$ satisfies {\rm(P)} then
$$
\forall \psi\in W^{1,p}(\Omega;\RR^3)\ \ \displaystyle\overline{I}(\psi)=\overline{I}_{\rm aff}(\psi)=\int_\Omega \Q W_0(\nabla\psi(x))dx.
$$
\end{CorollaryA}

\subsection{Application 2: ``weak-Determinant Constraint"}\ 

The following condition on $W$ is compatible with (w-DC).
\begin{itemize}
\item[(D)] $\exists\alpha,\beta>0\ \forall F\in\MM^{N\times N}\ \big(|\det F|\geq\alpha\then W(F)\leq\beta(1+|F|^p)\big)$.
\end{itemize}
One can prove that 
$$
{\rm(D)}\then\exists c>0\ \forall F\in\MM^{N\times N}\ \Z W(F)\leq c(1+|F|^p)
$$
(see \cite{oah-jpm08b,oah-jpm09}) which roughly means that the ``weak-Determinant Constraint" is $p$-ample. Applying Theorem B we obtain
\newtheorem*{CorollaryB}{\bf Corollary 2}
\begin{CorollaryB}
If $W$ satisfies {\rm(D)} then
$$
\forall \phi\in W^{1,p}(\Omega;\RR^N)\ \ \displaystyle\overline{I}(\phi)=\overline{I}_{\rm aff}(\phi)=\int_\Omega \Q W(\nabla\phi(x))dx.
$$
\end{CorollaryB}

\begin{proof}[\small\bf\em Proof of a part of Corollary 2]
\small Taking the first part of Theorem B$^\prime$ into account, it suffices to verify the following two points:
\begin{itemize}
\item[\SMALL$\blacklozenge$] (D) $\then$ $\exists c>0\ \forall F\in\MM^{N\times N}\ \R W(F)\leq c(1+|F|^p)$;
\item[\SMALL$\blacklozenge$] (D) $\then$ $\Z_\infty W<+\infty$,
\end{itemize}
which will give us the desired integral representation for $\overline{I}$. The first point is due to a lemma by Ben Belgacem (see \cite{benbelgacem96}, see also \cite{oah-jpm09}). For the second point, it is obvious that  $\Z_\infty W(F)<+\infty$ for all $F\in\MM^{N\times N}$ with $|\det F|\geq \alpha$. On the other hand, we have
\newtheorem*{small-lemma}{\small\bf Lemma}
\begin{small-lemma}[Dacorogna-Ribeiro \cite{daco-rib04} 2004, see also \cite{celada-perrotta98}]\ \hskip150mm

$\forall F\in\MM^{N\times N}\ \big(|\det F|<\alpha\then \exists\varphi\in W^{1,\infty}(Y;\RR^N)\ \ |\det(F+\nabla\varphi(x))|=\alpha\ \hbox{ p.p. dans }Y\big)$.
\end{small-lemma}
Hence, if $F\in\MM^{N\times N}$ is such that $|\det F|<\alpha$ then $\Z_\infty W(F)\leq\int_{Y}W(F+\nabla\varphi(x))dx$ with some $\varphi\in W^{1,\infty}(Y;\RR^N)$ given by the lemma above, and so $\Z_\infty W(F)\leq 2^p\beta(1+|F|^p+\|\nabla\varphi\|^p_{L^p})<+\infty$.
\end{proof}

\subsection{From $p$-ample to non-$p$-ample case}\ 

Because of the following theorem, none of the theorems of this section  can be directly used for dealing with ($\mathcal{P}_1$) under (s-DC).
\begin{Theorem}[Fonseca \cite{fonseca88} 1988]\ 

If $W$ satisfies {\rm(s-DC)} then{\rm:}
\begin{itemize}
\item[(F$_1$)] $\Q W$ is rank-one convex{\rm;}
\item[(F$_2$)] $\Q W(F)=+\infty$ if and only if $\det F\leq 0$ and $\Q W(F)\to+\infty$ as $\det F\to 0^+$.
\end{itemize}
\end{Theorem}
The assertion (F$_2$) roughly says  that the ``strong-Determinant Constraint" is not $p$-ample, i.e., $\Z_\infty W$ cannot be of $p$-polynomial growth, and so neither Theorem A nor Theorem B is consistent with (s-DC). From the assertion (F$_1$) we see that $\Q W\leq\R W$ which shows that $\R W$ cannot be of $p$-polynomial growth when combined with (F$_2$). Hence, the theorem of Ben Belgacem is not compatible with (s-DC).

\begin{Question}
Develop strategies for passing from $p$-ample to non-$p$-ample case.
\end{Question}


\section{3d-2d passage with determinant type constraints}

\subsection{Statement of the problem}\ 

Let $W:\MM^{3\times 3}\to[0,+\infty]$ be Borel measurable and $p$-coercive (with $p>1$) and, for each $\eps>0$, let $I_\eps:W^{1,p}(\Sigma_\eps;\RR^3)\to[0,+\infty]$ be defined by
$$
I_\eps(\phi):={1\over\eps}\int_{\Sigma_\eps}W(\nabla\phi(x,x_3))dxdx_3,
$$
where $\Sigma_\eps:=\Sigma\times]-{\eps\over 2},{\eps\over 2}[\subset\RR^3$ with $\Sigma\subset\RR^2$ Lipschitz, open and bounded, and a point of $\Sigma_\eps$ is denoted by $(x,x_3)$ with $x\in\Sigma$ and $x_3\in]-{\eps\over 2},{\eps\over 2}[$. The problem of 3d-2d passage is the following.
\begin{itemize}
\item[($\mathcal{P}_2$)] {\em Prove (or disprove) that
$$
\forall\psi\in W^{1,p}(\Sigma;\RR^3)\ \ \Gamma(\pi)\hbox{-}\lim_{\eps\to 0}I_\eps(\phi)=\int_\Sigma W_{\rm mem}(\nabla\psi(x))dx
$$ 
and find a representation formula for $W_{\rm mem}:\MM^{3\times 2}\to[0,+\infty]$.}
\end{itemize}
At the begining of the nineties, Le Dret and Raoult answered to ($\mathcal{P}_2$) in the case where $W$ is ``finite and without singularities" (see \S 2.3). Recently, we extended the Le Dret-Raoult theorem to the case where $W$ is compatible with (w-DC) and (s-DC) as Theorem C and Theorem D (see \S 2.4 and \S 2.5).

\subsection{The $\Gamma(\pi)$-convergence}\ 

The concept of $\Gamma(\pi)$-convergence was introduced Anzellotti, Baldo and Percivale in order to deal with dimension reduction problems in mechanics. Let $\pi=\{\pi_\eps\}_\eps$ be the family of $L^p$-continuous maps $\pi_\eps:W^{1,p}(\Sigma_\eps;\RR^3)\to W^{1,p}(\Sigma;\RR^3)$ defined by
$$
\displaystyle\pi_\eps(\phi):={1\over\eps}\int_{-{\eps\over 2}}^{\eps\over 2}\phi(\cdot,x_3)dx_3.
$$
\newtheorem*{Def}{\bf Definition}
\begin{Def}[Anzellotti-Baldo-Percivale \cite{anzellotti-baldo-percivale94} 1994]\ 

We say that $\{I_\eps\}_{\eps}$ $\Gamma(\pi)$-converge to $I_{\rm mem}$ as $\eps$ goes to zero, and we write 
$$
I_{\rm mem}=\Gamma(\pi)\hbox{-}\lim\limits_{\eps\to 0} I_\eps,
$$ 
if and only if 
$$
\forall\psi\in W^{1,p}(\Sigma;\RR^3)\quad\left(\Gamma(\pi)\hbox{-}\liminf\limits_{\eps\to 0} I_\eps\right)(\psi)=\left(\Gamma(\pi)\hbox{-}\limsup\limits_{\eps\to 0} I_\eps\right)(\psi)=I_{\rm mem}(\psi)
$$
with $\Gamma(\pi)\hbox{-}\liminf\limits_{\eps\to 0} I_\eps, \Gamma(\pi)\hbox{-}\limsup\limits_{\eps\to 0} I_\eps:W^{1,p}(\Sigma;\RR^3)\to[0,+\infty]$ respectively given by{\rm:}
\begin{itemize}
\item[\SMALL$\blacklozenge$] $\Gamma(\pi)\hbox{-}\liminf\limits_{\eps\to 0} I_\eps(\psi):=\inf\left\{\liminf\limits_{\eps\to 0}I_\eps(\phi_\eps):\pi_\eps(\phi_\eps)\stackrel{L^p}{\to}\psi\right\};$
\item[\SMALL$\blacklozenge$] $\Gamma(\pi)\hbox{-}\limsup\limits_{\eps\to 0} I_\eps(\psi):=\inf\left\{\limsup\limits_{\eps\to 0}I_\eps(\phi_\eps):\pi_\eps(\phi_\eps)\stackrel{L^p}{\to}\psi\right\}$.
\end{itemize}
\end{Def}
Anzellotti, Baldo and Percivale proved that their concept of $\Gamma(\pi)$-convergence is not far from that of $\Gamma$-convergence introduced by De Giorgi and Franzoni.  For each $\eps>0$, consider $\mathcal{I}_\eps:W^{1,p}(\Sigma;\RR^3)\to[0,+\infty]$ defined by
$$
\mathcal{I}_\eps(\psi):=\inf\Big\{I_\eps(\phi):\pi_\eps(\phi)=\psi\Big\}.
$$
\begin{Def}[De Giorgi-Franzoni \cite{degiorgi-franzoni75,degiorgi75} 1975]\ 

We say that $\{\mathcal{I}_\eps\}_{\eps}$ $\Gamma$-converge to $I_{\rm mem}$ as $\eps$ goes to zero, and we write 
$$
I_{\rm mem}=\Gamma\hbox{-}\lim\limits_{\eps\to 0} \mathcal{I}_\eps,
$$
if and only if
$$
\forall\psi\in W^{1,p}(\Sigma;\RR^3)\quad\left(\Gamma\hbox{-}\liminf\limits_{\eps\to 0} \mathcal{I}_\eps\right)(\psi)=\left(\Gamma\hbox{-}\limsup\limits_{\eps\to 0} \mathcal{I}_\eps\right)(\psi)=I_{\rm mem}(\psi)
$$
with $\Gamma\hbox{-}\liminf\limits_{\eps\to 0} \mathcal{I}_\eps, \Gamma\hbox{-}\limsup\limits_{\eps\to 0} \mathcal{I}_\eps:W^{1,p}(\Sigma;\RR^3)\to[0,+\infty]$ respectively given by{\rm:}
\begin{itemize}
\item[\SMALL$\blacklozenge$] $\Gamma\hbox{-}\liminf\limits_{\eps\to 0} \mathcal{I}_\eps(\psi):=\inf\left\{\liminf\limits_{\eps\to 0}\mathcal{I}_\eps(\psi_\eps):\psi_\eps\stackrel{L^p}{\to}\psi\right\};$
\item[\SMALL$\blacklozenge$] $\Gamma\hbox{-}\limsup\limits_{\eps\to 0} \mathcal{I}_\eps(\psi):=\inf\left\{\limsup\limits_{\eps\to 0}\mathcal{I}_\eps(\psi_\eps):\psi_\eps\stackrel{L^p}{\to}\psi\right\}.$
\end{itemize}
\end{Def}
The link between $\Gamma(\pi)$-convergence and $\Gamma$-convergence is given by the following lemma. 
\begin{Lemma}[see \cite{anzellotti-baldo-percivale94}]\ 

$
I_{\rm mem}=\Gamma(\pi)\hbox{-}\lim\limits_{\eps\to 0} I_\eps$ if and only if $I_{\rm mem}=\Gamma\hbox{-}\lim\limits_{\eps\to 0} \mathcal{I}_\eps
$.
\end{Lemma}


\subsection{$\Gamma(\pi)$-convergence of $I_\eps$: finite case}\ 

Let $W_0:\MM^{3\times 2}\to[0,+\infty]$ be defined by
$$
W_0(\xi):=\inf\limits_{\zeta\in\RR^3}W(\xi\mid\zeta).
$$
\begin{Theorem}[Le Dret-Raoult \cite{ledret-raoult93,ledret-raoult95} 1993]\ \hskip150mm

If $W$ is continuous and 
$\exists c>0\ \forall F\in\MM^{3\times 3}\ W(F)\leq c(1+|F|^p)$
then
$$
\forall\psi\in W^{1,p}(\Sigma;\RR^3)\ \ \Gamma(\pi)\hbox{-}\lim\limits_{\eps\to 0}I_\eps(\psi)=\int_\Sigma\Q W_0(\nabla\psi(x))dx.
$$
\end{Theorem}
Although the Le Dret-Raoult theorem is compatible neither with (w-DC) nor (s-DC) it established a suitable variational framework to deal with dimensional reduction problems : it is the point of departure of many works on the subject.

\subsection{$\Gamma(\pi)$-convergence of $I_\eps$: ``weak-Determinant Constraint"}\ 

By using the Le Dret-Raoult theorem we can prove the following result.

\newtheorem*{ThC}{\bf Theorem C}
\begin{ThC}[see \cite{oah-jpm06,oah-jpm09}]\ 

Assume that
\begin{itemize}
\item[(D)] $\exists\alpha,\beta>0\ \forall F\in\MM^{3\times 3}\ \big(|\det F|\geq\alpha\then W(F)\leq\beta(1+|F|^p)\big)$.
\end{itemize}
Then
$$
\forall\psi\in W^{1,p}(\Sigma;\RR^3)\quad\Gamma(\pi)\hbox{-}\displaystyle\lim\limits_{\eps\to 0}I_\eps(\psi)=\int_\Sigma\Q W_0(\nabla\psi(x))dx.
$$
\end{ThC}
\begin{proof}[\small\bf\em Outline of the proof] 

\small$\blacktriangleright$ As the $\Gamma(\pi)$-limit is stable by substituting $I_\eps$ by its relaxed functional $\overline{I}_\eps$, i.e., $\overline{I}_\eps:W^{1,p}(\Sigma_\eps;\RR^3)\to[0,+\infty]$ given by
$$
I_\eps(\phi):=\inf\left\{\liminf\limits_{n\to+\infty}I_\eps(\phi_n):\phi_n\stackrel{L^p}{\to}\phi\right\}
={1\over\eps}\inf\left\{\liminf\limits_{n\to+\infty}\int_{\Sigma_\eps}W(\nabla\phi_n)dxdx_3:\phi_n\stackrel{L^p}{\to}\phi\right\},
$$
it suffices to prove that
$$
\forall\psi\in W^{1,p}(\Sigma;\RR^3)\quad\Gamma(\pi)\hbox{-}\displaystyle\lim\limits_{\eps\to 0}\overline{I}_\eps(\psi)=\int_\Sigma\Q W_0(\nabla\psi(x))dx. 
$$

\small$\blacktriangleright$ As $W$ satisfies (D) it is $p$-ample (see \S 1.6), and so by Theorem A-B we have 
$$
\forall\eps>0\ \forall \phi\in W^{1,p}(\Sigma_\eps;\RR^3)\quad\overline{I}_\eps(\phi)={1\over\eps}\int_{\Sigma_\eps}\Q W(\nabla\phi(x,x_3))dxdx_3
$$
with $\Q W=\Z_\infty W$ (which is of $p$-polynomial growth and so continuous by the property (2) of Fonseca).

\smallskip

\small$\blacktriangleright$ Applying the Le Dret-Raoult theorem we deduce that
$$
\forall \psi\in W^{1,p}(\Sigma;\RR^3)\quad\Gamma(\pi)\hbox{-}\lim_{\eps\to 0}\overline{I}_\eps(\psi)=\int_{\Sigma}\Q[\Q W]_0(\nabla\psi(x))dx
$$
with $[\Q W]_0:\MM^{3\times 2}\to[0,+\infty]$ given by
$$
[\Q W]_0(\xi):=\inf_{\zeta\in\RR^3}\Q W(\xi\mid\zeta).
$$

\small$\blacktriangleright$ Finally, we prove that $\Q[\Q W]_0=\Q W_0$, and the proof is complete.
\end{proof}

Theorem C highlights the fact that the concept of $p$-amplitude has a ``nice" behavior with respect to the $\Gamma(\pi)$-convergence. More generally, let $\{\pi_\eps\}_\eps$ be a family of $L^p$-continuous maps $\pi_\eps$ from $W^{1,p}(\Sigma_\eps;\RR^m)$ to $W^{1,p}(\Sigma;\RR^m)$, where $\Sigma_\eps\subset\RR^N$ (resp. $\Sigma\subset\RR^k$ with $k\in\NN^*$) is a bounded open set,  let $\{W_\eps\}_{\eps}$ be an uniformly $p$-coercive family of measurable integrands $W_\eps:\MM^{m\times N}\to[0,+\infty]$ and, for each $\eps>0$, let $I_\eps, \Q I_\eps:W^{1,p}(\Sigma_\eps;\RR^m)\to[0,+\infty]$ be respectively defined by
\begin{itemize}
\item[\SMALL$\blacklozenge$] $\displaystyle I_\eps(\phi):=\int_{\Sigma_\eps}W_\eps(\nabla\phi(x))dx;$
\item[\SMALL$\blacklozenge$] $\displaystyle \Q I_\eps(\phi):=\int_{\Sigma_\eps}\Q W_\eps(\nabla\phi(x))dx$.
\end{itemize}
The following theorem says that the $\Gamma(\pi)$-limit is stable by substituting $I_\eps$ by $\Q I_\eps$ whenever every $W_\eps$ is $p$-ample.
\begin{Theorem}[see \cite{oah-jpm09}]\ 

Assume that{\rm:}
\begin{itemize}
\item[\SMALL$\blacklozenge$] $\forall\eps>0$ $W_\eps$ is $p$-ample{\rm;}
\item[\SMALL$\blacklozenge$] $\exists I_0:W^{1,p}(\Sigma;\RR^m)\to[0,+\infty]$ $\Gamma(\pi)\hbox{-}\lim\limits_{\eps\to 0}\Q I_\eps=I_0$.
\end{itemize}
Then $\Gamma(\pi)\hbox{-}\lim\limits_{\eps\to 0} I_\eps=I_0$.
\end{Theorem}
\begin{proof}
\small As every $W_\eps$ is $p$-ample, from Theorem A-B we deduce that $\overline{I}_\eps=\Q I_\eps$ for all $\eps>0$. On the other hand, as every $\pi_\eps$ is $L^p$-continuous, it is easy to see that  $\Gamma(\pi)$-$\liminf_{\eps\to0} I_\eps=\Gamma(\pi)$-$\liminf_{\eps\to0} \overline{I}_\eps$ and $\Gamma(\pi)$-$\limsup_{\eps\to0} I_\eps=\Gamma(\pi)$-$\limsup_{\eps\to0} \overline{I}_\eps$, and the theorem follows.
\end{proof}

\subsection{$\Gamma(\pi)$-convergence of $I_\eps$: ``strong-Determinant Constraint"}\ 

The following theorem gives an answer to ($\mathcal{P}_2$) in the framework of nonlinear elasticity (it is consistent with (s-DC)) in the same spirit as the theorem of Ball in 1977 (see \cite{ball77}). It is the result of several works on the subject: mainly, the attempt of Percivale in 1991 (see \cite{percivale91}), the rigorous answer to ($\mathcal{P}_2$) by Le Dret and Raoult in the $p$-polynomial growth case (see \cite{ledret-raoult93,ledret-raoult95}) and especially the substantial contributions of Ben Belgacem (see \cite{benbelgacem96,benbelgacem97,benbelgacem00}).

\newtheorem*{ThD}{\bf Theorem D}
\begin{ThD}[see \cite{oah-jpm08b,oah-jpm09}]\ 

Assume that{\rm:}
\begin{itemize}
\item[(D$_0$)] $W$ is continuous\hbox{\rm ;} 
\item[(D$_1$)] $W(F)=+\infty\iff\det F\leq 0;$
\item[(D$_2$)] $\forall\delta>0\ \exists c_\delta>0\ \forall F\in\MM^{3\times 3}\big(\det F\geq\delta\then W(F)\leq c_\delta(1+|F|^p)\big)$.
\end{itemize}
Then
$$
\forall\psi\in W^{1,p}(\Sigma;\RR^3)\quad\Gamma(\pi)\hbox{-}\displaystyle\lim\limits_{\eps\to 0}I_\eps(\psi)=\int_\Sigma\Q W_0(\nabla\psi(x))dx.
$$
\end{ThD}
{{\small\bf\em Outline of the proof.} \small$\blacktriangleright$ It is easy to see that if $W$ satisfies (D$_0$), (D$_1$) and (D$_2$) then:
\begin{itemize}
\item[(P$_0$)] $W_0$ is continuous;
\item[(P$_1$)] $\forall\alpha>0\ \exists \beta_\alpha>0\ \forall \xi\in\MM^{3\times 2}\big(|\xi_1\land\xi_2|\geq\alpha\then W_0(\xi)\leq \beta_\alpha(1+|\xi|^p)\big)$.
\end{itemize}
In particular, $W_0$ satisfies (P) since clearly (P$_1$) implies (P).

\smallskip

\small$\blacktriangleright$ Let $\mathcal{I}, \overline{\mathcal{I}}, \overline{\mathcal{I}}_{\rm diff_*}:W^{1,p}(\Sigma;\RR^3)\to[0,+\infty]$ be respectively defined by:
\begin{itemize}
\item[\SMALL$\blacklozenge$] $\displaystyle\mathcal{I}(\psi):=\int_\Sigma W_0(\nabla\psi(x))dx$;
\item[\SMALL$\blacklozenge$] $\overline{\mathcal{I}}(\psi):=\inf\left\{\liminf\limits_{n\to+\infty}\mathcal{I}(\psi_n):\psi_n\stackrel{L^p}{\to}\psi\right\}$;
\item[\SMALL$\blacklozenge$] $\overline{\mathcal{I}}_{\rm diff_*}(\psi):=\inf\left\{\liminf\limits_{n\to+\infty}\mathcal{I}(\psi_n):C^1_*(\overline{\Sigma};\RR^3)\ni\psi_n\stackrel{L^p}{\to}\psi\right\}$,
\end{itemize}
where $C^1_*(\overline{\Sigma};\RR^3)$ is the set of $C^1$-immersions from $\overline{\Sigma}$ to $\RR^3$, i.e., $$C^1_*(\overline{\Sigma};\RR^3):=\Big\{\psi\in C^1(\overline{\Sigma};\RR^3):\forall x\in\overline{\Sigma}\ \partial_1\psi(x)\land\partial_2\psi(x)\not= 0\Big\}.$$ As $W_0$ satisfies (P), by Corollary 1 we have
$$
\forall\psi\in W^{1,p}(\Sigma;\RR^3)\quad\overline{\mathcal{I}}(\psi)=\int_\Sigma\Q W_0(\nabla\psi(x))dx.
$$
On the other hand, we can prove the following two lemmas.
\begin{small-lemma}
$\overline{\mathcal{I}}\leq\Gamma(\pi)\hbox{-}\liminf\limits_{\eps\to 0}I_\eps$.
\end{small-lemma}

\begin{small-lemma}
If {\rm(D$_0$)}, {\rm(D$_1$)} and {\rm(D$_2$)} hold then $\Gamma(\pi)\hbox{-}\limsup\limits_{\eps\to 0}I_\eps\leq \overline{\mathcal{I}}_{\rm diff_*}$.
\end{small-lemma}
Hence, it suffices to prove that $\overline{\mathcal{I}}_{\rm diff_*}\leq \overline{\mathcal{I}}$.

\smallskip

\small$\blacktriangleright$ Let $\overline{\mathcal{I}}_{\rm aff_{\rm li}},\R\mathcal{I},\overline{\mathcal{R I}},\overline{\mathcal{R I}}_{\rm aff_{\rm li}}:W^{1,p}(\Sigma;\RR^3)\to[0,+\infty]$ be respectively defined by:
\begin{itemize}
\item[\SMALL$\blacklozenge$] $\displaystyle \overline{\mathcal{I}}_{\rm aff_{\rm li}}(\psi):=\inf\left\{\liminf_{n\to+\infty}\mathcal{I}(\psi_n):\Aff_{\rm li}(\Sigma;\RR^3)\ni\psi_n\stackrel{L^p}{\to}\psi\right\}$;
\item[\SMALL$\blacklozenge$] $\displaystyle \R\mathcal{I}(\psi):=\int_\Sigma\R W_0(\nabla\psi(x))dx$;
\item[\SMALL$\blacklozenge$] $\displaystyle \overline{\mathcal{R I}}(\psi):=\inf\left\{\liminf_{n\to+\infty}\mathcal{R}\mathcal{I}(\psi_n):\psi_n\stackrel{L^p}{\to}\psi\right\}$;
\item[\SMALL$\blacklozenge$] $\displaystyle \overline{\mathcal{R I}}_{\rm aff_{\rm li}}(\psi):=\inf\left\{\liminf_{n\to+\infty}\mathcal{R}\mathcal{I}(\psi_n):\Aff_{\rm li}(\Sigma;\RR^3)\ni\psi_n\stackrel{L^p}{\to}\psi\right\}$
\end{itemize}
with $\Aff_{\rm li}(\Sigma;\RR^3):=\big\{\psi\in\Aff(\Sigma;\RR^3):\psi\hbox{ is locally injective}\big\}$. As $\overline{\mathcal{R I}}\leq\overline{\mathcal{I}}$, a way for proving $\overline{\mathcal{I}}_{\rm diff_*}\leq \overline{\mathcal{I}}$ is to establish the following  three inequalities:
\begin{itemize}
\item[\SMALL$\blacklozenge$] $\overline{\mathcal{I}}_{\rm diff_*}\leq \overline{\mathcal{I}}_{\rm aff_{\rm li}}$;
\item[\SMALL$\blacklozenge$] $\overline{\mathcal{I}}_{\rm aff_{\rm li}}\leq \overline{\mathcal{R I}}_{\rm aff_{\rm li}}$;
\item[\SMALL$\blacklozenge$] $\overline{\mathcal{R I}}_{\rm aff_{\rm li}}\leq \overline{\mathcal{R I}}$.
\end{itemize}
The first inequality follows by using the fact that $W_0$ satisfies (P$_0$) and (P$_1$) together with the following lemma.
\begin{Lemma}[Ben Belgacem-Bennequin \cite{benbelgacem96} 1996, see also \cite{oah-jpm09}]\ 

For all $\psi\in\Aff_{\rm li}(\Sigma;\RR^3)$ there exists $\{\psi_n\}_{n\geq 1}\subset C^1_*(\overline{\Sigma};\RR^3)$ such that{\rm:}
\begin{itemize}
\item[\SMALL$\blacklozenge$] $\psi_n\stackrel{W^{1,p}}{\to}\psi;$
\item[\SMALL$\blacklozenge$] $\exists\delta>0\ \forall x\in\overline{\Sigma}\ \forall n\geq 1\ |\partial_1\psi_n(x)\land\partial_2\psi_n(x)|\geq\delta$.
\end{itemize}
\end{Lemma}
The second inequality is obtained by exploiting the Kohn-Strang representation of $\R W_0$ (see \cite{benbelgacem96}, see also \cite{oah-jpm09}). Finally, we establish  the next inequality by combining the following two lemmas.
\begin{Lemma}[Ben Belgacem \cite{benbelgacem96} 1996, see also \cite{oah-jpm09}]\ 

If $W_0$ satisfies {\rm (P$_0$)} and {\rm(P$_1$)} then{\rm:}
\begin{itemize}
\item[\SMALL$\blacklozenge$] $\mathcal{R} W_0$ is continuous{\rm;}
\item[\SMALL$\blacklozenge$] $\exists c>0\ \forall \xi\in\MM^{3\times 2}\ \R W_0(\xi)\leq c(1+|\xi|^p)$.
\end{itemize}
\end{Lemma}
\begin{Lemma}[Gromov-{\`E}lia{\v{s}}berg \cite{gromov-eliashberg71} 1971, see also \cite{oah-jpm09}]\ 

$\Aff_{\rm li}(\Sigma;\RR^3)$ is strongly dense in $W^{1,p}(\Sigma;\RR^3)$. \hfill$\square$
\end{Lemma}}

\begin{Question}
Try to simplify the proof of Theorem D as follows{\rm:} first, approximate $W$ satisfying {\rm(D$_0$)}, {\rm(D$_1$)} and {\rm(D$_2$)} or maybe weaker conditions compatible with {\rm(s-DC)} by a supremum of $p$-ample integrands $W_\delta$ satisfying {\rm(D)} with $\alpha,\beta>0$ which can depend on $\delta$, then, apply Theorem C to each $W_\delta$, and finally, pass to the limit as $\delta$ goes to zero.
\end{Question}

\bibliographystyle{acm}

\end{document}